\documentclass[12pt, reqno]{amsart}
\usepackage{amsfonts}
\usepackage{bbm}
\usepackage{amscd,amsfonts}
\usepackage{amssymb, eucal, amsfonts, amsmath, xypic, latexsym, tikz}
\usepackage{pifont}
\usepackage{mathrsfs,color}
\usepackage{amsthm,indentfirst,bm,fancyhdr,dsfont}
\usepackage{graphicx}
\usepackage[all]{xy}
\usepackage[CJKbookmarks=true]{hyperref}

\usepackage{mathrsfs}
\usepackage{amsmath}
\usepackage{amssymb}
\usepackage{hyperref}

\newtheorem{theorem}{Theorem}[section]
\newtheorem{lemma}[theorem]{Lemma}
\newtheorem{prop}[theorem]{Proposition}

\newtheorem{corollary}[theorem]{Corollary}
\theoremstyle{definition}
\newtheorem{conj}[theorem]{Conjecture}
\newtheorem{defn}[theorem]{Definition}

\newtheorem{remark}[theorem]{Remark}

\numberwithin{equation}{section}

\def\ggg{\mathfrak{g}}
\def\zzz{\mathfrak{z}}

\def\gl{\mathfrak{gl}}

\def\ggg{\mathfrak{g}}

\def\ppp{\mathfrak{p}}
\def\qqq{\mathfrak{q}}

\def\hhh{\mathfrak{h}}
\def\aaa{\mathfrak{a}}
\def\bbb{\mathfrak{b}}
\def\nnn{\mathfrak{n}}

\def\calb{\mathcal{B}}

\def\caln{\mathcal{N}}

\def\cz{\mathcal {Z}}

\def\bk{\mathbf{k}}

\def\Fr{\textsf{Fr}}
\def\tsb{\textsf{B}}

\def\bbz{\mathbb{Z}}
\def\bbq{\mathbb{Q}}

\def\bk{\mathbf{k}}

\def\bo{{\bar 1}}
\def\bz{{\bar 0}}

\def\sft{\textsf{t}}

\def\tsx{\textsf{X}}
\def\tsy{\textsf{Y}}
\def\tsyy{\mathscr{Y}}

\def\ad{\mathsf{ad}}

\def\Ad{\text{Ad}}

\def\GL{\text{GL}}

\def\id{\mathsf{id}}

\def\tsb{\textsf{B}}
\def\tsbo{{}^1\textsf{B}}
\def\tsf{\textsf{F}}
\def\tsg{\textsf{G}}
\def\tsfd{\textsf{d}}

{\vskip-\lastskip\medskip
  \noindent
  }%
{\qed\par\medskip
  }

\def\GL{\text{\rm GL}}

\def\Ad{\textsf{Ad}}

\begin{document}

\title[The $1^{\text{st}}$ sKW conjecture for strange classical Lie superalgebras]
{Modular representations of strange classical  Lie superalgebras and the first super Kac-Weisfeiler conjecture
}

\author{ Ye Ren, Bin Shu, Fanlei Yang and An Zhang}

\address{YR: School of Mathematical Sciences, Ministry of Education Key Laboratory of Mathematics and Engineering Applications \& Shanghai Key Laboratory of PMMP,  East China Normal University, Shanghai 200241, China} \email{523922184@qq.com}
\address{BS: School of Mathematical Sciences, Ministry of Education Key Laboratory of Mathematics and Engineering Applications \& Shanghai Key Laboratory of PMMP,  East China Normal University, Shanghai 200241, China} \email{bshu@math.ecnu.edu.cn}
\address{FY: School of Mathematical Sciences, Ministry of Education Key Laboratory of Mathematics and Engineering Applications \& Shanghai Key Laboratory of PMMP,  East China Normal University, Shanghai 200241, China}\email{51215500044@stu.ecnu.edu.cn}
\address{AZ: School of Mathematical Sciences, Ministry of Education Key Laboratory of Mathematics and Engineering Applications \& Shanghai Key Laboratory of PMMP,  East China Normal University, Shanghai 200241, China}\email{52215500003@stu.ecnu.edu.cn}





\thanks{This work is partially supported by the National Natural Science Foundation of China (Grant Nos. 12071136,  12271345), supported in part by Science and Technology Commission of Shanghai Municipality (No. 22DZ2229014).}

\begin{abstract} Suppose $\ggg=\ggg_\bz+\ggg_\bo$ is a Lie superalgebra of queer type or periplectic type over an algebraically closed field $\bk$ of characteristic $p>2$.
In this article, we initiate preliminarily to  investigate modular representations of periplectic Lie superalgebras 
 and then verify  the first super Kac-Weisfeiler conjecture on the maximal  dimensions of irreducible modules  for $\ggg$ proposed by the second-named author in \cite{Shu} where the conjecture is targeted at all finite-dimensional restricted Lie superalgebras over $\bk$, and already proved to be true for basic classical Lie superalgebras and completely solvable restricted Lie superalgebras.
	\end{abstract}
\maketitle
\setcounter{tocdepth}{1}\tableofcontents
\section*{Introduction}

\subsection{} Recall that there are two types of finite-dimensional classical simple Lie algebras over complex numbers: basic classical series, and strange classical series  (see \cite{Kac}). The difference of both lies in that the former ones admit non-degenerate even invariant symmetric forms, but the latter ones do not have such forms. So this difference leads to the difference between their representation theory.  Parallel to the case of  the complex numbers, one can define all classical Lie superalgebras in odd prime characteristic.  On the other side, those classical Lie superalgebras  naturally arises from algebraic supergroups of Chevalley type in odd prime characteristic (see \cite{FG1})

Since the works on irreducible representations of algebraic supergroups in odd characteristic, especially Wang-Zhao's work  focusing on modular Lie superalgebras of basic classical type, the study of irreducible representations of finite-dimensional Lie superalgebras in odd characteristic has a quite big progress (see \cite{B06}, \cite{BKl}, \cite{BKu}, \cite{CSW}, \cite{LS}, \cite{Shi}, \cite{SW}, \cite{WZ1}, \cite{WZ2}, \cite{ZengS1}, \cite{ZengS2}, \cite{Zhao}, {\sl{etc.}}). Nevertheless, their irreducible modules are not well-understood. Especially, the study for strange classical series has less progress. For example, although the Kac-Weiffeiler property becomes already a remarkable nature of  modular representations of basic classical Lie superalgebras according to Wang-Zhao's work \cite{WZ1}, but it seems to be difficult to judge with this property in strange classical series. The verification has neither  been finished for queer type (see \cite{WZ1, WZ2}), and nor been done at all for periplectic type. It is worthwhile mentioning that  there are some remarkable works  on modular representations for queer Lie superalgebras and algebraic supergroups (see \cite{B06}, \cite{BKl} for supegroups and \cite{WZ2} for Lie superalgeras).


\subsection{First super  Kac-Weisfeiler conjecture} As a super counterpart of the first Kac-Weisfeiler conjecture on the maximal dimensions of irreducible modules for restricted Lie algebras (see \cite{Kac2}),   the second-named author of the present paper proposed  a conjecture in  \cite{Shu} which says   for any finite-dimensional restricted Lie superalgebras $\ggg$ over an algebraically closed field $\bk$ of characteristic $p>2$, the maximal dimension of irreducible $\ggg$-modules  is  $p^{\frac{b_0}{2}}2^{\lfloor\frac{b_1}{2}\rfloor}$
(see \S\ref{sec: skw conj} for the meanings of   notations).
 This conjecture is a super version of the first Kac-Weisfeiler conjecture on irreducible modules of  restricted Lie algebras (see \cite{Kac2}). It was verified in the cases when the Lie superalgebras are basic classical Lie superalgebras and completely solvable restricted Lie superalgebras (see \cite{Shu}).

\subsection{Main results} One of the main purposes of this paper is to verify this conjecture for Lie superalgebras of strange classical series (see Theorem \ref{thm: kw conj for strange} and Corollary  \ref{cor: 5.2 p} and \ref{prop: first KW for derirved q}). Before that, we will preliminarily investigate irreducible representations of periplectic Lie superalgebras in the modular case, which is also one purpose of this  paper. As to the part of modular representations of periplectic Lie superalgebras, we obtain that all baby Verma modules for regular semisimple $p$-characters are irreducible (see Propositions \ref{prop: reg ss}).   This means that irreducible modules of maximal dimensions happen ``densely" in certain sense of algebraic geometry (see Remark \ref{rem: regular ss dense}). A more non-trivial result further says that the baby Verma modules in  regular nilpotent cases are also irreducible (see Proposition \ref{prop: reg nil}). These results are certainly important for the future research.

\section{Preliminaries}
Throughout the paper, the notions of vector spaces (resp. modules, subalgebras) means vector superspaces (resp. super-modules, super-subalgebras). For simplicity,
we will often omit the adjunct word ``super."  All vector spaces are defined over $\bk$ which is an algebraically closed field of characteristic $p>2$. For superspace $V=V_\bz+V_\bo$, we will mention the super-dimension of $V$ which means $\underline{\dim} V:=(\dim V_\bz|\dim V_\bo)$, in the meanwhile we mention the dimension of $V$ which means $\dim_\bk V:=\dim V_\bz+\dim V_\bo$. Throughout the paper, all Lie (super)algebras are finite-dimensional.  For the basic notions and notations, we follow  the following references \cite{CW}, \cite{FG1}, \cite{Kac}, \cite{Man}.

\subsection{Restricted Lie superalgebras}\label{restricted}
A Lie superalgebra $\ggg=\ggg_\bz\oplus\ggg_\bo$ is called a restricted one if $\ggg_\bz$ is a restricted Lie algebra and $\ggg_\bo$ is a restricted module of $\ggg_\bz$. Modular representations of classical Lie superalgebras are much related to the ones of reductive Lie algebras, the latter of which have been understood well (see \cite{FP88}, \cite{Jan97}, \cite{Kac2}, \cite{Pre},  \cite{SF}, \cite{WK}, and \cite{Zass}, {\sl{etc.}}).

Denote by $\cz(\ggg)$ the center of $U(\ggg)$, i.e. $\cz(\ggg):=\{u\in U(\ggg)\mid \ad x(u)=0\;\; \forall x\in \ggg\}$.
By the definition of a restricted Lie superalgebra, the whole $p$-center of $U(\ggg_0)$ falls in $\cz$, this is to say,  $x^p-x^{[p]}\in \cz, \forall x\in \ggg_{\bar 0}$. We denote by $\cz_\Fr$ the  $p$-center (or Frobenius center).
  Fix a basis $\{x_1,\cdots,x_s\}$ of $\ggg_\bz$ and a basis $\{y_1,\cdots,y_t\}$ of $\ggg_\bo$.
Set $\xi_{i}=x_{i}^{p}-x_{i}^{[p]}, i=1,\ldots, s$. The $p$-center
$\cz_\Fr$ is a polynomial ring
$\bk[\xi_{1},\ldots,\xi_{s}]$ generated by $\xi_{1},\ldots, \xi_{s}$.

By the PBW theorem, one easily knows that the enveloping
superalgebra $U(\ggg)$ is a free module over $\cz_\Fr$ with basis
\[x_{1}^{a_{1}}\cdots x_{s}^{a_{s}}y_{1}^{b_{1}}\cdots y_{t}^{b_{t}},
 0\leq a_i\leq p-1, \; b_{j}\in\{0,1\}\mbox{ for }i=1,\cdots,s, j=1,\cdots,t .\]

\subsection{Reduced enveloping algebras of restricted Lie superalgebras}\label{reduced}
Suppose $V$ is an irreducible $U(\ggg)$-module. By the above argument, $x^p-x^{[p]}$ for $x\in \ggg_\bz$, falls in the center of $\cz_\bz$. By Schur's Lemma we have that each $x^p-x^{[p]}$ for $x\in \ggg_\bz$ acts by a scalar $\chi(x)^p$ for some $\chi\in {\ggg_\bz}^*$. Such  $\chi$ is called the $p$-character of  $V$. Suppose $\chi\in {\ggg_\bz}^*$ is given, denote by $I_\chi$ the ideal of $U(\ggg)$ generated by the even central elements $x^p-x^{[p]}-\chi(x)^p$.  More generally, we can say that  a $U(\ggg)$-module $M$ is  a $\chi$-reduced module for a given $\chi\in \ggg^*$ if  for any $x\in \ggg_\bz$, $x^p-x^{[p]}$ acts by a scalar $\chi(x)^p$. All $\chi$-reduced modules for a given $\chi\in\ggg^*$ constitute a full subcategory of the $U(\ggg)$-module category.
The quotient algebra $U_\chi(\ggg):=U(\ggg)\slash I_\chi$ is called the reduced enveloping superalgebra of $p$-character $\chi$.  Then the $\chi$-reduced module category of $\ggg$ coincides with the $U_\chi(\ggg)$-module category. If $\hhh$ is a restricted Lie subalgebra of $\ggg$, we often take use of  $\chi|_{\hhh_{\bz}}$ for $\chi\in \ggg^*_\bz$ when considering the $U_\chi(\ggg)$-module category and its  objects  induced from $\hhh$-modules. By abuse of notations, we will simply write $\chi|_{\hhh_{\bz}}$ as $\chi$.

By the PBW theorem, the superalgebra $U_\chi(\ggg)$ has a basis \[x_{1}^{a_{1}}\cdots x_{s}^{a_{s}}y_{1}^{b_{1}}\cdots y_{t}^{b_{t}},
 0\leq a_{i}\leq p-1; b_{j}\in\{0,1\}\mbox{ for }i=1,\cdots s; j=1,\cdots,t.\]
And $\dim U_\chi(\ggg)=p^{\dim \ggg_\bz}2^{\dim \ggg_\bo}$.

 The reduced enveloping algebra corresponding to $\chi=0$ is  $U_0(\ggg)$.  We call it the restricted enveloping algebra of $\ggg$. The modules of $U_0(\ggg)$ are called restricted modules of $\ggg$. The following observation is clear.

\begin{lemma} The dimension of an irreducible module of a restricted Lie algebra $\ggg$ is not greater  than $p^{\dim \ggg_\bz}2^{\dim \ggg_\bo}$.
\end{lemma}

\subsection{Question and conjecture} There is a natural question:
\begin{align*}
 \text{what is the maximal dimensions of irreducible modules for $\ggg$?}
\end{align*}
In \cite{Shu}, the author proposed  a conjecture that the maximal dimensions is
\begin{align}
p^{\frac{b_0}{2}}2^{{\lfloor\frac{b_1}{2}\rfloor}}=\max_{\theta\in \ggg_\bz^*}p^{\frac{b^\theta_0}{2}}2^{{\lfloor\frac{b^\theta_1}{2}\rfloor}}
\end{align}
where the precise meaning of the above notations can be seen in the paragraph prior to  (\ref{eq: 1 KW version}). This conjecture  can be regarded a super version of the first Kac-Weisfeiler conjecture which turned out true for basic classical Lie superalgebras and completely solvable restricted Lie superalgebras in \cite{Shu}. The main purpose is to verify this conjecture for Lie superalgebras of  strange classical  series.

\section{First look into modular representations of periplectic Lie superalgebras}
\subsection{Periplectic Lie superalgebras} In this section, we will present some basic facts on periplectic Lie superalgebras, and preliminarily investigate their modular representations.

\subsubsection{} For a given positive integer $n$, we have
\begin{align}\label{eq: perp n}
\tilde\ppp(n):=\{X:=\left( \begin{array}{cc} A  & B
\cr C & -A^\sft\end{array}\right)\in\gl(n|n)  \mid B=B^\sft, C=-C^\sft\}.
\end{align}

\begin{remark}\label{rem: for p} Set $\ppp(n)=[\tilde\ppp(n), \tilde\ppp(n)]$ which we call a derived periplectic Lie superalgebra.  Then $\tilde\ppp(n)=\ppp(n)+\bk\tsfd$ where $\tsfd$ is the $X$ in (\ref{eq: perp n}) with $B=0=C$ and $A$ equal to the identity matrix.
In the concluding section, we will separately check  the first super Kac-Weifeisler conjecture for derived periplectic and queer Lie superalgebras.

Recall that their representation theory over the complex field has been studied by various authors (see \cite{BDE}, \cite{CC}, \cite{Chen}, \cite{ES1}, \cite{ES2}, \cite{Gor},   \cite{Ser}, {\sl{etc.}}).
%
\end{remark}

%

\subsubsection{} From now on to the end of this section, we set $\ggg=\tilde\ppp(n)$.
 Then $\ggg$ is a Lie subalgebra of $\gl(n|n)$ with $\ggg=\ggg_\bz+\ggg_\bo$ where $\ggg_\bz=\ggg\cap \gl(n|n)$, and $\ggg_\bo=\ggg\cap \gl(n|n)$.
Precisely,
$$\ggg_\bz=\{\left( \begin{array}{cc} A  & 0
\cr 0 & -A^\sft\end{array}\right)  \mid A\in \gl(n)\}$$
and $\ggg_\bo=\ggg_{-1}\oplus \ggg_{1}$ with
\begin{align*}
&\ggg_1=\{\left( \begin{array}{cc} 0  & B
\cr 0 & 0\end{array}\right) \mid B\in \gl(n) \text{ satisfying } B=B^\sft\}\cr
&\ggg_{-1}=\{\left( \begin{array}{cc} 0  & 0
\cr C & 0\end{array}\right)  \mid C\in \gl(n)\text{ satisfying } C=-C^\sft\}.
\end{align*}
Clearly, $\ggg_\bz$ is isomorphic to $\gl(n)$ as Lie algebras, and $\dim\ggg_\bo=n^2$ with $\dim\ggg_{\pm1}$ equal to ${1\over 2}(n^2\pm n)$ respectively.
Set $\hhh\subset\ggg_\bz$  to be the standard maximal torus consisting of toral elements
$$ H=\left( \begin{array}{cc} \textsf{diag}(h_1,\ldots,h_n)  & 0
\cr 0 & \textsf{diag}(-h_1,\ldots,-h_n)\end{array}\right)\in\ggg_\bz \text{ with }(h_1,\ldots,h_n)\in \bk^n.$$
So we can talk about the triangular decomposition of $\ggg_\bz$ which is $\ggg_\bz=\nnn_\bz^-\oplus \hhh \oplus \nnn_\bz^+$, corresponding to the standard one of $\gl(n)$. Set $\bbb_\bz=\hhh\oplus \nnn_\bz^+$.

\subsection{Standard Cartan subalgebra and root system}\label{sec: root system of p}
The above $\hhh$ becomes a Cartan subalgebra of $\ggg$, called the standard Cartan subalgebra. With respect to $\hhh$, we can list the root systems $\Phi_0$, $\Phi_{\pm1}=\Phi(\ggg_{\pm1},\hhh)$ of $\ggg$, $\ggg_\bz$, and $\ggg_{\pm1}$ respectively:
$\Phi_0=\Phi_0^+\cup (-\Phi_0^+)$ with $\Phi_0=\{\epsilon_i-\epsilon_j\mid 1\leq i<j\leq n\}$,
$\Phi_1=\{\epsilon_i+\epsilon_j\mid 1\leq i\leq j\leq n\}$, and
$\Phi_{-1}=\{-\epsilon_i-\epsilon_j\mid 1\leq i<j\leq n\}$.
The root system $\Phi$ of $\ggg$ is equal to  $\Phi_0\cup\Phi_1\cup \Phi_{-1}$. There is a canonical  positive root system $\Phi^+_0\cup \Phi_1$ corresponding to $\caln^+:=\nnn_\bz^+\oplus \ggg_1$, and the negative root system $(-\Phi_0^+)\cup \Phi_{-1}$ corresponding to $\caln^-:=\nnn_\bz^-\oplus \ggg_{-1}$. This positive root system has a fundamental root system $$\Pi=\Pi_0\cup \{2\epsilon_n\}$$ where $\Pi_0=\{\epsilon_i-\epsilon_{i+1}\mid i=1,\ldots,n-1\}$ is the standard fundamental root system of $\gl(n)$.
Set $\Upsilon=\{2\epsilon_i\mid i=1,\ldots,n\}$ and
$$\mathcal{C}:=\mathcal{C}(\Pi\cup \Upsilon)$$
the complement of $\Pi\cup \Upsilon$ in $\Phi^+$.

\subsection{} Furthermore, $\ggg=\hhh\oplus \sum_{\alpha\in \Phi} \ggg_\alpha$ where all root space $\ggg_\alpha$ are one-dimensional.
Denote by $E_{k,l}$ the $(k,l)$th elementary matrix of size $2n\times 2n$, which means the matrix with all entries equal to zero unless the $(k,l)$-entry equal to $1$.
Then the generating root vector of $\ggg_{\alpha}$ is  easily expressed as,
\begin{align*}
E_{i,j}-E_{n+j,n+i}\in \ggg_\alpha\subset \ggg_\bz\cr
E_{j,i}-E_{n+i,n+j}\in \ggg_{-\alpha}\subset \ggg_\bz
\end{align*}
for $\alpha=\epsilon_i-\epsilon_j$ with $1\leq i<j\leq n$,
\begin{align*}
E_{i,j+n}+E_{j, n+i}\in \ggg_\alpha\subset\ggg_1 \cr
E_{i+n,j}-E_{j+n, i} \in \ggg_{-\alpha}\subset\ggg_{-1}
\end{align*}
for $\alpha=\epsilon_i+\epsilon_j, 1\leq i< j\leq n$,
and $E_{i,i+n}\in \ggg_\alpha\subset \ggg_1$ for $\alpha=2\epsilon_i$, $i=1,\ldots,n$.

For $\alpha=\epsilon_i-\epsilon_j\in\Phi^+_0$ or $\alpha=\epsilon_i+\epsilon_j \in \Phi_1$ with   $1\leq i<j\leq n$, we have $[\ggg_\alpha,\ggg_{-\alpha}]\subset \hhh$.
 choose $X_\alpha\in \ggg_\alpha$ and $X_{-\alpha}\in\ggg_{-\alpha}$ such that $[X_\alpha, X_{-\alpha}]=H_\alpha$ with $(H_\alpha,H_\alpha)=2$ where $(A,B):=\textsf{tr}(A,B)$ for $A,B\in \ggg_\bz$.


By a straightforward computation, the following observation is clear.

\begin{lemma} Let $\ggg=\tilde\ppp(n)$. 
The Lie superalgebra $\ggg$  has a gradation of
the additive group $\{-1,0,1\}\cong \bbz_3$ with $\ggg_\bo=\ggg_1\oplus\ggg_{-1}$ and $\ggg_\bz=\ggg_0$.
\end{lemma}

\subsection{The $p$-structure of $\ggg=\tilde\ppp(n)$}
Recall $\ggg_\bz\cong \gl(n)$ is a restricted Lie algebra with canonical $p$-structure via the $p$-mapping on the basis elements  $X_\alpha^{[p]}=0$, $H_{i}^{[p]}=H_{i}$  for
\begin{align*}
&\{H_{i}=E_{ii}-E_{n+i,n+i}\in\hhh\mid i=1,\ldots,n-1\}\cup\cr
&\cup \{
X_\alpha=E_{ij}-E_{n+j,n+i}\in \ggg_\alpha\mid \alpha=\epsilon_i-\epsilon_j, 1\leq i\ne j\leq n\}.
\end{align*}
Furthermore, $\ggg_\bo$ is a restricted $\ggg_\bz$-module. This is to say, $\ggg$ is a restricted Lie superalgebra.

\subsection{Conjugation of $\ggg$}\label{sec: cong}
For $\aaa=\gl(n)$ or $\mathfrak{sl}(n)$, and correspondingly $\ggg_0=\tilde\ppp(n)_\bz$ or $\ppp(n)_\bz$, consider the isomorphism of Lie algebras $\phi: \aaa\rightarrow \ggg_\bz$ mapping $A$ onto $\left( \begin{array}{cc} A  & 0
\cr 0 & -A^\sft\end{array}\right)$. For
$g\in\GL(n)$
set $\bar{g}:=\left( \begin{array}{cc} g  & 0
\cr 0 & g^{-\sft}\end{array}\right)$.  Then the conjugation of $\ggg_0$ via $\GL(n)$ gives rise to the conjugation of $\overline\ggg$ as below:
\begin{align*}
\textsf{Ad}g(X):= \bar{g} X
{\bar{g}}^{-1}
\end{align*}
where $X=\left( \begin{array}{cc} A  & B
\cr C & -A^\sft\end{array}\right)\in {\ggg}$.
In this way, any element of $\GL(n)$ gives rise to an automorphism of $\bar\ggg$ compatible with the $p$-structure, i.e. $\textsf{Ad}g$ commutates with the $p$-mapping $[p]$.

Note that there is an $\GL(n)$-equivariant isomorphism between  $\ggg_\bz$ and $\gl(n)$ for $\ggg=\tilde\ppp(n)$. 
And  $\gl(n)$  and $\gl(n)^*$ 
are $\GL(n)$-equivariant isomorphic. By abuse of notations, we still take use of notation $\Ad(g)$ for the coadjoint action of $g\in \GL(n)$ on $\ggg_\bz^*$. Hence, the coadjoing action from $\GL(n)$ makes
\begin{align}\label{eq: Ad iso of red env}
U_\chi(\ggg)\cong U_{\textsf{Ad}g(\chi)}(\ggg)
\end{align}
for any $g\in \GL(n)$.

 For $\ggg=\tilde\ppp(n)$, 
 there is a $\GL(n)$-equivariant isomorphism between $\ggg_\bz$ and $\ggg_\bz^*$. 
 Then any $\chi\in \ggg_\bz^*$ has a Jordan standard form which means $\chi(\nnn_\bz^+)=0$. By (\ref{eq:  Ad iso of red env}), without loss of generality we may suppose
 \begin{align}\label{hyp: ann n+}
 \chi \text{  satisfies }\chi(\nnn_\bz^+)=0.
 \end{align}

\subsection{Example $\tilde\ppp(2)$} In this subsection, we precisely describe irreducible modules of $\tilde\ppp(2)$. Recall
\begin{align}\label{eq: example p(2)}
\tilde\ppp(2)=\{\left( \begin{array}{cc} A  & B
\cr C & -A^\sft\end{array}\right)\in\gl(2|2)  \mid B=B^\sft, C=-C^\sft\}. \}.
\end{align}
Especially $\tilde\ppp(2)=[\tilde\ppp(2), \tilde\ppp(2)]\oplus \bk\zzz$ for $\zzz=\text{diag}(1,1,-1,-1)$ which is in the center of $\tilde\ppp(2)$. In the following, we simply write $\tilde\ppp$ for  $\tilde\ppp(2)$, and $\ppp$ for $[\tilde\ppp(2),\tilde\ppp(2)]$. Then $\tilde\ppp_\bz=\ppp_\bz\oplus \bk\zzz$, and  $\ppp_\bz=\bk E+\bk F+\bk H$ is canonically  isomorphic to $\mathfrak{sl}(2)$ where under the isomorphism, $E,F, H$ respectively correspond to the basis elements of $\mathfrak{sl}(2)$:
$$e=\left( \begin{array}{cc} 0  & 1
\cr 0 & 0\end{array}\right);\; f=\left( \begin{array}{cc} 0  & 0
\cr 1 & 0\end{array}\right);\; \text{ and }h=\left( \begin{array}{cc} 1  & 0
\cr 0 & -1\end{array}\right).$$  And $\tilde\ppp_\bo=\ppp_\bo=(\bk X\oplus \bk Z_1\oplus \bk Z_2) \oplus \bk Y$ where $X,Z_1,Z_2$ and $Y$ are the standard root vectors corresponding to $\epsilon_1+\epsilon_2$, $2\epsilon_1$, $2\epsilon_2$, and $-(\epsilon_1+\epsilon_2$ respectively. Especially, $Y$ is of the form $\left( \begin{array}{cc} A  & B
\cr C & -A^\sft\end{array}\right)$ in (\ref{eq: example p(2)}) with $A=0=B$, and $C=\left( \begin{array}{cc} 0  & 1
\cr -1 & 0\end{array}\right)$.

Note that the trace function in $\mathfrak{gl}(2)$ gives rise to a non-degenerate bilinear form which is $\text{GL}(2)$-equivariant. Hence  the $\text{GL}(2)$-conjugacy class of any given  $\chi\in \tilde\ppp_\bz^*$ admits a representative element $\theta$ satisfying
$\theta(E)=0$. And $U_\chi(\ggg)$ is isomorphic to $U_\theta(\ggg)$. So we might as well suppose $\chi(E)=0$ in the following without loss of generality. In this case, by the classical results (see \cite{Jan97}) or direct computations,  irreducible $U_\chi(\tilde\ppp_\bz)$-module can be described as:
 \begin{itemize}
 \item[(1)] The $p$-dimensional spaces $L^0_\chi(\lambda):=U_\chi(\tilde\ppp_\bz)\otimes_{U_\chi(\bk E\oplus\bk H\oplus\bk\zzz)}\bk_\lambda$ with $\lambda\in \Lambda(\chi):=\{\lambda\in (\bk H\oplus\bk \zzz)^*\mid \lambda(\diamond)^p-\lambda(\diamond)=\chi(\diamond)^p, \text{ for }\diamond=H \text{ or }\zzz\}$ if either $\chi(\ppp_\bz)\ne0$  or $\chi(\ppp_\bz)=0$ with $\lambda(H)\ne0$. Here $E$ acts trivially on the one-dimensional space $\bk_\lambda$, and $\diamond$ acts on $\bk_\lambda$ by the scalar function $\lambda$;
\item[(2)] The one-dimensional spaces $L^0_\chi(c):=\bk_c$ if $\chi(\ppp_\bz)=0$ while $\lambda(H)=0$ and  $c=\lambda(\zzz)$. Here $\ppp_\bz$ acts trivially on $\bk_c$, and $\zzz$ acts on $\bk_c$ by scalar $c$.
    \end{itemize}
By unifying the notations with $\lambda\in \Lambda(\chi)$, we can define Kac modules
$$K_\chi(\lambda)=U_\chi(\tilde\ppp)\otimes_{U_\chi(\tilde\ppp_\bz\oplus \tilde\ppp_1)}L^0(\lambda)$$
where  $L^0(\lambda)$ is regarded an irreducible $U_\chi(\tilde\ppp_\bz\oplus \tilde\ppp_1)$-module with trivial $\tilde\ppp_1$-action. By a direct computation, we have the following observation.

\begin{lemma}  \begin{itemize}
\item[(1)] For any $\chi\in \ppp_\bz^*$, $\chi$ is $\text{GL}(2)$-conjugate to $\theta\in \ppp_\bz^*$ with $\theta(E)=0$, and $U_\chi(\ggg)$ is isomorphic to $U_\theta(\ggg)$.
\item[(2)] For any $\chi\in \tilde\ppp_\bz^*$ with $\chi(E)=0$,  and $\lambda\in \Lambda(\chi)$,
 the following statements hold.
\begin{itemize}
\item[(2.1)] If $\chi(\ppp_\bz)\ne 0$,  or  $\chi(\ppp_\bz)=0$ but $\lambda(H)\ne 0$,
then all $K_\chi(\lambda)$ are irreducible and $2p$-dimensional.

 \item[(2.2)] If $\chi(\ppp_\bz)=0$ with  $\lambda(H)=0$, then  $K_\chi(c)$ has one-dimensional irreducible submodule $Y\otimes\bk_c$, denoted by $L_\chi(c)$.
\item[(2.3)] Conversely, any irreducible $U_\chi(\tilde\ppp)$-module is isomorphic to $K_\chi(\lambda)$ in the case of (2.1), or to $L_\chi(c)$ in the case of (2.2).
\end{itemize}
\end{itemize}
\end{lemma}

\subsection{A general look at irreducible representations of $\ggg$} By the arguments in \S\ref{sec: cong}, we only need to consider $p$-characters annihilating $\nnn_\bz^+$ without loss of any generality. Fix such  a $p$-character $\chi$ with $\chi(\nnn_\bz^+)=0$.
Set $\Lambda(\chi)=\{\lambda\in \hhh^*\mid \lambda(H)^p-\lambda(H^{[p]})=\chi(H)^p\;\forall H\in \hhh\}$. Then
$$\Lambda(\chi)=\{\lambda\in \hhh^*\mid \lambda(H_i)^p-\lambda(H_i)=\chi(H_i)^p, i=1,\ldots,{n}\}.$$
Then for any $\lambda \in \Lambda(\chi)$ we have a baby Verma module $Z_\chi^0(\lambda):=U_\chi(\ggg_0)\otimes_{U_\chi(\bbb_\bz)}\bk_\lambda$ of $U_\chi(\ggg_0)$. Any simple module of $U_\chi(\ggg_0)$ is an irreducible quotient of $Z_\chi^0(\lambda)$. Fix an irreducible module $L_\chi^0(\lambda)$ with $\lambda\in \Lambda(\chi)$ for $U_\chi(\ggg_0)$
Set $\ggg_{\geq 0}:=\ggg_0\oplus \ggg_1$. Then $\ggg_{\geq0}$ is obviously a restricted Lie subalgebra of $\ggg$.
And $Z_\chi^0(\lambda)$ can be regarded a $U_\chi(\ggg_{\geq0})$-module  with trivial $\ggg_1$-action because $\ggg_1$ is actually a nilpotent ideal of $\ggg_{\geq 0}$. On the other side, consider $\calb:=\bbb_\bz\oplus\ggg_1$ which is obviously a solvable restricted subalgebra of $\ggg$ with nilpotent radical $\caln^+=\nnn_\bz^+\oplus \ggg_1$. Furthermore, any irreducible module over $U_\chi(\calb)$ is one-dimensional $\bk_\lambda$ for $\lambda\in \Lambda(\chi)$. Similarly, one can consider $\caln^{-}=\nnn_\bz^-\oplus \ggg_{-1}$ which is a restricted subalgebra of $\ggg$.

Consider the following generalized baby Verma module
$$Z_\chi(\lambda):=U_\chi(\ggg)\otimes_{U_\chi(\calb)}\bk_\lambda. $$
Then we have
\begin{align}\label{eq: baby Ver mod}
Z_\chi(\lambda)&= U_\chi(\ggg)\otimes_{U_\chi(\ggg_{\geq0})}(U_\chi(\ggg_{\geq0})
\otimes_{U_\chi(\calb)}\bk_\lambda)            \cr
&= U_\chi(\ggg)\otimes_{U_\chi(\ggg_{\geq0})}Z_\chi^0(\lambda)
\end{align}
where $Z_\chi^0(\lambda)$ is regarded a $U_\chi(\ggg_{\geq0})$ with trivial $\ggg_1$-action. This makes sense because $\ggg_1$ is actually a nilpotent ideal in $\ggg_{\geq 0}$. Note that $[\ggg_{-1}, \ggg_{-1}]=0$, hence the subalgebra generated by $\ggg_{-1}$ in $U_\chi(\ggg)$ is isomorphic to $\bigwedge^\bullet\ggg_{-1}$.
 As a vector space,  $Z_\chi(\lambda)$ consequently equals to $U_\chi(\caln^{-})\otimes \bk_\lambda=\bigwedge^\bullet \ggg_{-1}\otimes_\bk Z_\chi^0(\lambda)$.

The natural $\bbz$-gradation of the exterior algebra $\bigwedge^\bullet\ggg_{-1}=\sum_{k=0}^N \bigwedge^{k}\ggg_{-1}$ gives rise to the gradation of $Z_\chi(\lambda)$
$$Z_\chi(\lambda)=\sum_{k=0}^{N}Z_\chi(\lambda)_{[k]}$$ where  $N:=\dim\ggg_{-1}={{n(n-1)}\over 2}$, and $Z_\chi(\lambda)_{[k]}=\bigwedge^{k}\ggg_{-1}\otimes Z^0_\chi(\lambda)$. Clearly, each graded subspace $Z_\chi(\lambda)_{[k]}$ is a $U_\chi(\ggg_\bz)$-module.
In particular, $Z_\chi(\lambda)_{[0]}=1\otimes Z^0_{\chi}(\lambda)$ and $Z_{\chi}(\lambda)_{[N]}=Y\otimes Z^0_\chi(\lambda)$.

\begin{defn}\label{def: regular ss}
Call $\chi\in \ggg_\bz^*$ is regular semisimple if there exits $g\in \GL(n)$ such that $\Ad(g)\chi\in \hhh^*$ with $\Ad(g)\chi(\nnn_\bz^\pm)=0$ and $\Ad(g)\chi(H_i-H_j)\ne 0$ for any $1\leq i<j\leq n\}$.
\end{defn}
\begin{remark}\label{rem: regular ss dense} Note that $\ggg$ and $\ggg^*$ are $\text{GL}(n)$-equivariant linear isomorphic.   The set of regular semisimple elements
   in $\ggg^*$ forms a Zariski dense subset of $\ggg^*$ (see \cite[Corollary 2.1.13]{Coll}).
\end{remark}

{
Set $\tsx=\prod_{1\leq i<j\leq n} X_{\epsilon_i+\epsilon_j}$ and $\tsy=\prod_{1\leq i<j\leq n} X_{-(\epsilon_i+\epsilon_j)}$.  Let $v$ be a generator of one-dimensional vector space $1\otimes\bk_\lambda$ which is then a generator of $Z_\chi(\lambda)$ over $U_\chi(\ggg)$.
 By \cite[Lemma 3.1]{Ser},
 \begin{align}\label{eq: XYv}
\tsx\tsy v=\Omega(\lambda) v
\end{align}
where $\Omega(\lambda)=\prod_{1\leq i<j\leq n}(\lambda_i-\lambda_j+j-i-1)$, and $\lambda_i=\lambda(H_i)$ for $i=1,\ldots,n$.

\begin{lemma}\label{lem: top irr} Suppose $\chi$ is regular semisimple, then the above top homogeneous space $Z_{\chi}(\lambda)_{[N]}$ is an irreducible $U_\chi(\ggg_\bz)$-module.
\end{lemma}
\begin{proof} Under the assumption, $\chi$ becomes a regular semisimple $p$-character of $\ggg_\bz\cong \gl(n)$.  By a classical result (see \cite{FP88} or \cite{Jan97}), $Z^0_\chi(\lambda)$ is irreducible for any $\lambda\in \Lambda(\chi)$.

Consider $Z_{\chi}(\lambda)_{[N]}=\tsy\otimes Z^0_\chi(\lambda)$. Set $\delta=-\sum_{1\leq i<j\leq n}(\epsilon_i+\epsilon_j)$.  We want to show that $\tsy\otimes Z^0_\chi(\lambda)$ is isomorphic to $Z^0_\chi(\lambda+\delta)$, as a $U_\chi(\ggg_\bz)$-module. Note that any $Z^0_\chi(\mu)$ is a free $U(\nnn^-_\bz)$-module of rank one with generator  $v_\mu:=1\otimes 1$ in $1\otimes \bk_\mu$. Then $\tsy\otimes Z^0_\chi(\ggg_\bz)$ admits a one-dimensional  $U_\chi(\bbb_\bz)$-module generated by $\tsy\otimes v_\lambda$ because $[\tsy,\nnn_\bz^\pm]=0$. Furthermore, $\tsy\otimes v_\lambda$ is of weight $\lambda+\delta$.
Hence, we can define  a homomorphism of $U_\chi(\ggg_\bz)$-modules
 $$\phi: Z^0_\chi(\lambda+\delta)\rightarrow \tsy\otimes Z^0_\chi(\lambda),\; \text{ with } uv_{\lambda+\delta} \mapsto \tsy\otimes uv_\lambda,;\forall u\in U_\chi(\nnn^-_\bz).$$
Clearly, $\phi$ is surjective. Both of $Z^0_\chi(\lambda+\delta)$ and $\tsy\otimes Z^0_\chi(\lambda)$ have the same dimension $p^{\dim\nnn^-_\bz}$. Hence $\phi$ becomes an isomorphism.
\end{proof}
}

\begin{prop}\label{prop: reg ss}
\begin{itemize}

\item[(1)] Any irreducible $U_\theta(\ggg)$-module is a quotient of $Z_\chi(\theta)$ for any $\theta\in \ggg_\bz^*$ with $\theta(\nnn_\bz^+)=0$.

\item[(2)] If $\chi$ is regular semisimple, then $Z_\chi(\lambda)$ is an irreducible $U_\chi(\ggg)$-module.

\end{itemize}
\end{prop}

\begin{proof} (1) is clear. For (2),
 we keep the notations in (\ref{eq: XYv}).
By (\ref{eq: Ad iso of red env}), we might as well assume the the regular semisimple $\chi\in\ggg_\bz^*$ lies in $\hhh^*$.
 Under the assumption again, $Z^0_\chi(\lambda)$ is irreducible for any $\lambda\in \Lambda(\chi)$.

   Note that $\lambda_i-\lambda_j\in \Lambda(\chi)$ which satisfies the equation in $\bk$
$$(\lambda_i-\lambda_j)^p-(\lambda_i-\lambda_j)=h_{ij}^p$$
where $h_{ij}=\chi(H_{i}-H_j)\ne 0$ by the assumption that $\chi$ is regular semisimple.
Hence $\lambda_i-\lambda_j+j-i-1$ is nonzero for all $1\leq i<j\leq n$. Consequently, $\Omega(\lambda)\ne 0$.

Now we claim that for any nonzero $w\in Z_\chi(\lambda)$, the submodule $W$ generated $w$ over $U_\chi(\ggg)$ coincides with $Z_\chi(\lambda)$. Hence $Z_\chi(\lambda)$ is irreducible.  In order to prove this claim, we first observe that $\tsy$ is a generator of one-dimensional space $\bigwedge^{\dim\ggg_{-1}}\ggg_{-1}$. And $Z_\chi(\lambda)$ is a free module of rank one over $U_\chi(\caln^-)$. Hence by multiplying  a suitable element from $\bigwedge^\bullet(\ggg_{-1})$ on $w$, one has that

\begin{align*}
\text{
there is a nonzero vector }w':=\tsy u_\tsy v\in W \text{ for some }u_\tsy\in U_\chi(\nnn_\bz^-).
 \end{align*}
 It is deduced that the $U_\chi(\ggg_\bz)$-submodule $W_0$ generated by $w'$ is contained in $W$.  By Lemma \ref{lem: top irr}, $W_0$ coincides with $\tsy Z_\chi^0(\lambda)$. Hence $\tsy v\in W$. Furthermore, $\tsx \tsy v=\Omega(\lambda)v\in W$ and $\Omega(\lambda)\ne 0$. We finally have $v\in W$. Consequently $W$ coincides with $Z_\chi(\lambda)$.

The proof is completed.
\end{proof}

\subsection{Regular nilpotent representations} Say that $\chi\in\ggg_\bz^*$ is a nilpotent $p$-character if the corresponding element $e\in\ggg_\bz\cong \gl(n)$ is nilpotent under the $\GL(n)$-equivariant linear isomorphism between $\gl(n)$ and $\gl(n)^*$. This is to say, $\chi(\bbb)=0$ for $\bbb=\hhh\oplus \nnn_\bz^+$.
The representation theory of $U_\chi(\ggg_\bz)$ with $\chi$ being nilpotent plays a critically important role in the whole representation theory of $\ggg_\bz$, due to a Morita equivalence theorem (see \cite[Theorem 3.2]{FP88}). We further say that a nilpotent $p$-character $\chi$ is regular nilpotent if the coadjoint orbit $\GL(n).\chi$ has the greatest dimension in all nilpotent orbits.  According to the theory of  nilpotent orbits, all regular nilpotent elements lie in the same orbit with representative element $\theta$ satisfying $\theta(X_{-\alpha})=1$ for all $\alpha\in \Pi_0$ and $\theta(X_\beta)=0$ for all $\beta\in \Phi_0\backslash \Pi_0$ (see \cite[\S4.14, \S6.7]{Jan3}).

In the following, for any given regular nilpotent $p$-character $\chi$ we might as well suppose that $\chi$ satisfies
\begin{align}\label{eq: reg nilp}
\chi(X_{-\alpha})=1 \text{ for  }\alpha\in \Pi_0 \text{ and } \chi(X_\beta)=0 \text{ for }\beta\in \Phi_0\backslash\Pi_0
 \end{align}
 without loss of generality because a $\GL(n)$-conjugation of $p$-characters gives  rise to the isomorphism of the corresponding reduced enveloping algebras.
   \begin{prop}\label{prop: reg nil}  Suppose  $\chi\in\mathfrak{g}_{\bz}^*$ is a regular nilpotent $p$-character as in (\ref{eq: reg nilp}), and $\lambda\in \Lambda(\chi)$. Then  $Z_\chi(\lambda)$ is irreducible.
   \end{prop}

   \begin{proof} We proceed the arguments by different steps.

   (1) By the same argument around in (\ref{eq: baby Ver mod}), it can be shown that $Z_\chi(\lambda)= U_\chi(\ggg)\otimes_{U_\chi(\ggg_{\geq0})}Z_\chi^0(\lambda)$. Note that $Z_\chi^0(\lambda)$ is irreducible $U_\chi(\ggg_\bz)$-module because $\chi$ is regular nilpotent (see \cite[Theorem 4.2]{FP88} or \cite{Jan97}). Consequently, $Z_\chi^0(\lambda)$ is irreducible $U_\chi(\ggg_{\geq0})$-module.

   (2) Due to (1), by the same arguments as in the proof of Lemma \ref{lem: top irr} we have that
     $Z_{\chi}(\lambda)_{[N]}=\tsy\otimes Z^0_\chi(\lambda)$ is an irreducible $U_\chi(\ggg_\bz)$-module. On the other side, by  the exterior product property of $\bigwedge\ggg_{-1}$,  it is easily shown that for any given nonzero submodule $M$ of $Z_\chi(\lambda)$, $M\cap Z_\chi(\lambda)_{[N]}\ne 0$. So $M$ contains $\tsy\otimes v_\lambda$, where $v_\lambda$ is a generator $1\otimes 1\in 1\otimes \bk_\lambda$ of $Z^0_\chi(\lambda)$.
In order to show the proposition, it suffices for us to prove that $M$ contains the canonical generator $1\otimes v_\lambda$ of $Z_\chi(\lambda)$.

   (3) Keep the notations in \S\ref{sec: root system of p}. Similarly to the arguments in \cite{Ser}, we make a total order $\prec$ in  $\mathcal{C}$: for any different two roots $\alpha_k\in\{\epsilon_{i_k}\pm \epsilon_{j_k}\}$ with $1\leq i_k<j_k\leq n$, $k=1,2$, we define $\alpha_1\succ \alpha_2$ if $i_1>i_2$, or $i_1=i_2$ but $\alpha_2-\alpha_1\in\Phi^+$. Then we have
   $$\mathcal{C}=\{\beta_1\succ\beta_2\succ\cdots\succ\beta_\ell\}$$
   with $\ell:=\#\mathcal{C}$.
   Note that we already have $\tsy\otimes v_\lambda\in M$, and $M$ is a nonzero submodule of $Z_\chi(\lambda)$. So we have $\prod_{\beta\in \Phi_0^+\backslash \Pi_0} X_{-\beta}^{p-1} \tsy\otimes v_\lambda\in M$. Furthermore, by the property of Lie brackets in $\ggg$, we can change the order of factors for $\prod_{\beta\in \Phi_0^+\backslash \Pi_0} X_{-\beta}^{p-1} \tsy$, having  $\tsyy\otimes v_\lambda\in M$ for
$$\tsyy=X_{-{\beta_\ell}}^{\pi(\beta_\ell)}\cdots X_{-{\beta_1}}^{\pi(\beta_1)}$$
where $\pi(\beta)=p-1$ if $-\beta\in \Phi^-_0$, and $1$ if $-\beta\in \Phi_{-1}$.
    Set $\tsyy_k:=X_{-\beta_\ell}^{\pi(\beta_\ell)}\cdots X_{-\beta_k}^{\pi(\beta_k)}$ for $k=1,\ldots,\ell$. Naturally, $\tsyy_1=\tsyy$ and $\tsyy_\ell=X_{-\beta_\ell}^{\pi(\beta_\ell)}$.

   (4) Set
   \begin{align*}
   Z_k=\begin{cases}X_{-(\epsilon_i-\epsilon_{i+1})}^{p-1}
   X_{\epsilon_{i+1}+\epsilon_j}
   &\text{ if } \beta_k=(\epsilon_i+\epsilon_j)\in \Phi_{1}\cr
   X_{-(\epsilon_i-\epsilon_{i+1})}X^{p-1}_{\epsilon_{i+1}-\epsilon_j} &\text{ if }\beta_k=\epsilon_i-\epsilon_j\in \Phi_0^+\backslash\Pi_0
   \end{cases} \text{ for }i<j.
   \end{align*}
   Note that for any $\alpha\in \Pi_0$, $X^p_{-\alpha}=1$ in $U_\chi(\ggg)$, and $v_\lambda$ is annihilated by any root vectors of positive roots. We have that
   \begin{align}\label{eq: comput-1}
   Z_k X_{-\beta_k}^{\pi(\beta_k)}\otimes v_\lambda=1\otimes v_\lambda.
   \end{align}
   With aid of the forthcoming Lemma \ref{lem: 0.2},  we inductively have
   \begin{align*}
   Z_\ell\cdots Z_1 (\tsyy_1\otimes v_\lambda)&=Z_\ell\cdots Z_1(\tsyy_2 X_{-\beta_1}^{\pi(\beta_1)}\otimes v_\lambda)\cr
   &{\overset{Lemma \ref{lem: 0.2}(3)}{=}}Z_\ell\cdots Z_{2}\tsyy_2 (Z_1 X_{-\beta_1}^{\pi(\beta_1)}\otimes v_\lambda)\cr
   &{\overset{(\ref{eq: comput-1})}=}Z_\ell\cdots Z_{2}(\tsyy_2 \otimes v_\lambda)\cr
   &\cdots\cdots\cr
   &=Z_\ell(\tsyy_\ell\otimes v_\lambda)\cr
   &=1\otimes v_\lambda.
    \end{align*}
Hence $1\otimes v_\lambda\in M$, consequently $M=Z_\chi(\lambda)$.

Thus $Z_\chi(\lambda)$ is irreducible. The proof is completed.
   \end{proof}

\begin{lemma}\label{lem: 0.2} Keep the notations and assumptions as above. The following statements hold.
\begin{itemize}
		\item[(1)] Suppose $\alpha\in\{\epsilon_i\pm\epsilon_j\}, \beta\in\{\epsilon_s\pm\epsilon_t\}$ with $i<j$ and $s<t$ are two roots in $\mathcal{C}$ with $\alpha\succ\beta$. Then  $\beta-\alpha$ is either a positive root, or not any root at all. If $i=s$
(we say that $\alpha$ and $\beta$ have the same type $i=s$),
then $[X_{-\alpha},X_{-\beta}]=0$.
		\item[(2)] Suppose $\alpha\in \mathcal{C}$ with $\alpha\in \{\epsilon_i\pm\epsilon_j\}$. Then $\alpha=\alpha'+\alpha''$ for $\alpha'=\epsilon_i-\epsilon_{i+1}\in \Pi_0$ and $\alpha''\in \{\epsilon_{i+1}\pm\epsilon_j\}\subset\Phi^+$, satisfying
		$$[X_{\alpha''},\prod_{\beta\in \mathcal{C},\beta\prec\alpha}X^{\pi(\beta)}_{-\beta}]=0$$
and
		$$[X_{-\alpha'},\prod_{\beta\in \mathcal{C},\beta\prec\alpha}X^{\pi(\beta)}_{-\beta}]=0.$$

\item[(3)] $Z_k \tsyy_{k+1}=\tsyy_{k+1} Z_k$ for $k=1,\ldots, \ell-1$.
\end{itemize}
	\end{lemma}
	\begin{proof}
		(1) By the definition of order $\succ$ and the fact that $-2\epsilon_i$ is not any root, Part (1) follows.

		(2) 
Obviously, we have $\alpha\prec\alpha'\prec\alpha''$. 		Suppose $\beta-\alpha''$ is a root  for some $\beta\prec\alpha$. Then $\beta-\alpha''$ is a root in $\mathcal{C}$ since $\beta-\alpha''-\alpha'=\beta-\alpha$. Note that $\beta-\alpha''$ and $\beta$ have the same type. By simple calculation along with Part (1), we have
		$$[X_{\alpha''},\prod_{\beta\in \mathcal{C},\beta\prec\alpha}X^{\pi(\beta)}_{-\beta}]=\displaystyle
\sum_{\beta\prec\alpha}\displaystyle\sum_{1\leq j\leq \pi(\beta)}((\prod_{\gamma\prec\beta}X^{\pi(\gamma)}_{-\gamma})
X_{-\beta}^{j-1}[X_{\alpha''},X_{-\beta}]X_{-\beta}^{\pi(\beta)-j}
(\prod_{\gamma'\succ\beta}X_{-\gamma'}^{\pi(\gamma')}))=0$$
		because $X_{-\beta}^{\pi(\beta)+1}=0$ for any $\beta\in \mathcal{C}$.
		
		Similarly, suppose $\beta+\alpha'$ is a root for some $\beta\prec\alpha$. We have $\beta+\alpha'$ is a root in $\mathcal{C}$ has the same type with $\beta$ and  $\beta+\alpha'\prec\beta$. We also have
     	$$[X_{-\alpha'},\prod_{\beta\in
     \mathcal{C},\beta\prec\alpha}X^{\pi(\beta)}_{-\beta}]=
     \displaystyle\sum_{\beta\prec\alpha}\displaystyle\sum_{1\leq j\leq \pi(\beta)}((\prod_{\gamma\prec\beta}X^{\pi(\gamma)}_{-\gamma})
     X_{-\beta}^{j-1}[X_{-\alpha'},X_{-\beta}]X_{-\beta}^{\pi(\beta)-j}
     (\prod_{\gamma'\succ\beta}X_{-\gamma'}^{\pi(\gamma')}))=0.$$	

     (3) Note that for $\beta_k\in \mathcal{C}$ with $\beta_k\in\{\epsilon_i\pm\epsilon_j\}$, we have
     $\beta_k=\beta'+\beta''$ for $\beta'=\epsilon_i-\epsilon_{i+1}$ and $\beta''\in \{\epsilon_{i+1}\pm\epsilon_j\}$. Then Part (3) directly follows (2).
     		\end{proof}

\section{Recall of modular representations of queer Lie superalgebras}

Wang and Zhao initiated the study of  modular representation theory of Lie superalgebras  work \cite{WZ1} by formulating a general superalgebra analogue of the Kac-Weisfeiler (KW) conjecture and
establishing it for the basic classical Lie superalgebras. Furthermore, they studied modular representations of $Q(n)$ quite thoroughly.

\subsection{Definition} Recall that for the queer Lie superalgebra $\ggg =\qqq(n)$, it consists of matrices of the form
\begin{align}\label{eq: defn q}
\left( \begin{array}{cc} A  & B
\cr B & A\end{array}\right)\in\gl(n|n)
\end{align}
where $A$ and $B$ are arbitrary $n \times n$ matrices. Its even part $\ggg_\bz$ is consequently isomorphic to $\gl(n)$. The odd part $\ggg_\bo$ of $\qqq(n)$
is another isomorphic copy of $\gl(n)$ under the adjoint action of $\ggg_\bz$. Hence, the queer Lie superalgebra $\qqq(n)$ can be truly thought of super-analogue of the general linear Lie algebra. Recall that modular representations for type $Q$,  there are some remarkable works (see \cite{B06}, \cite{BKl} for supegroups and \cite{WZ2} Lie superalgeras).

\subsection{Standard Cartan subalgebras and roots}\label{sec: st csa and roots}
Let $\ggg=\qqq(n)$ from now on till the end of  this section.
Recall that  all Cartan subalgebras of $\ggg$ are conjugate to the standard Cartan subalgebra $\hhh =\hhh_\bz\oplus \hhh_\bo$ which consists of matrices (\ref{eq: defn q}) with both $A$
and $B$ diagonal. Such a Cartan subalgebra $\hhh$ will be called  the standard Cartan subalgebra. All Borel Lie subalgebras  are conjugate to the standard Borel subalgebra which  consists of matrices (\ref{eq: defn q}) with both $A$ and $B$ upper triangular. The roots
of $\ggg$ coincide with those of $\gl(n)$, which are by definition,  elements $\alpha \in \hhh^*_\bz$ satisfying that $\ggg_\alpha:= \{x \in\ggg\mid [H, x] = \alpha(H)x, \forall H\in\hhh_\bz\}$ is nonzero. More precisely,  let $\{\epsilon_i\}$ be a
basis of $\hhh^*_\bz$  dual to the standard basis $\{J_i\}$ of $\hhh_\bz$, where $J_i$ is of the form (\ref{eq: defn q}) with the $i$th diagonal entry of $A$ being 1 and $0$ elsewhere, then the roots are
$\Phi = \{\epsilon_i - \epsilon_j\mid 1 \leq i\ne j\leq n\}$. The toral element corresponding to $\epsilon_i-\epsilon_j$ is still denoted by $H_{\epsilon_i-\epsilon_j}$, which is actually equal to the matrix of form (\ref{eq: defn q}) with $B=0$ and $A=E_{ii}-E_{jj}$ (here the notations as in the proof of Proposition \ref{prop: reg ss}).
The super dimension of each root space is equal to $(1|1)$, in contrast to the $\gl(n)$ case.

For the fixed Borel subalgebra $\bbb$, we can define a system of positive roots which will be denoted by $\Phi^+$; the corresponding simple system is denoted by $\Pi$. Also let $\nnn^+ = \nnn^+_\bz\oplus \nnn^+_\bo$ (respectively $\nnn^-$) denote the Lie subalgebra of positive (respectively negative) root vectors.

\subsection{}\label{sec: conv on chi} Note that $\ggg_\bz = \gl(n)$ and then any $p$-character $\chi$ is $\text{GL}(n)$-conjugate to a $p$-character $\theta$ with $\theta(\nnn^+) =0$ and $U_\chi(\ggg)\cong
U_\theta (\ggg)$. So we can only need to consider $p$-characters $\chi$ with
\begin{align}\label{eq: ann n ev +}
\chi(\nnn_\bz^+) = 0
 \end{align}
 without loss of generality when considering  the question of maximal irreducible dimensions.  As in the case of periplectic $\ppp(n)$, a $p$-character $\chi\in\ggg_\bz^*$ is called semisimple or nilpotent if it is $\text{GL}(n)$-conjugate some $\theta\ggg_\bz^*$ with $\theta(\nnn^+_\bz\oplus\nnn^+_\bz)=0$, and $\chi$ is called nilpotent if it is $\text{GL}(n)$-conjugate to some $\theta\in \ggg_\bz^*$ with $\theta(\nnn^+_\bz\oplus\hhh_\bz)=0$.

\subsection{Queer baby Verma modules}\label{sec: strange baby}
In this subsection, we recall baby Verma modules introduced by Wang and Zhao in \cite{WZ2} which we will call queer baby Verma modules.

Fix a triangular decomposition $\ggg=\nnn^-\oplus \hhh\oplus \nnn^+$
and let $\bbb=\hhh\oplus \nnn^+$. We will only consider $p$-characters $\chi\in\ggg^*_\bz$ with $\chi(\nnn^+_\bz)=0$.

For $\lambda\in \hhh^*_\bz$ we may consider the symmetric bilinear form on $\hhh_\bo$ defined by
$\tsf_\lambda(Z_1|Z_2) := \lambda([Z_1,Z_2])$  $Z_1,Z_2\in \hhh_\bz$.
Now if $\hhh^\lambda_\bo$ is  a maximal isotropic subspace with respect to this bilinear form $\tsf$ and we let $\hat\hhh^\lambda_\bo$ be a complement of $\hhh^\lambda_\bo$ in $\hhh_\bo$ (i.e.  $\hhh_\bo=\hhh^\lambda_\bo\oplus \hat\hhh^\lambda_\bo$), we may extend  $\lambda$ to a one-dimensional representation $\bk_\lambda$ of $\hhh_\bz\oplus \hhh^\lambda_\bo$ by letting $\hhh^\lambda_\bo$ act trivially. Let $\chi\in \ggg_\bz^*$ be such that $\chi(\nnn^+)=0$. Set
$$\Lambda(\chi)=\{\lambda\in \hhh^*_\bz\mid
\lambda(H)^p - \lambda(H^{[p]}) = \chi(H)^p\;\; \forall H\in \hhh_\bz\}$$
which is equal to $\{(\lambda_1,\ldots, \lambda_n)\mid \lambda_i^p-\lambda_i= \chi(J_i)^p, 1\leq i \leq n\}$ for
where $\lambda_i = \lambda(J_i)$. The module $\bk_\lambda$ is a $U_\chi(\hhh_\bz\oplus \hhh_\bo')$-module if and only if $\lambda\in \Lambda_\chi$. Define the following $U_\chi(\hhh)$-module
$$V_\chi(\lambda) = U_\chi(\hhh)\otimes_{U_\chi(\hhh_\bz\oplus \hhh^\lambda_\bo)}\bk_\lambda$$
for $\lambda\in \Lambda(\chi)$. Then $V_\chi(\lambda)$ is irreducible over $U_\chi(\hhh)$, with dimension $2^{\dim\hat\hhh^\lambda_\bo}$.
This irreducible $U_\chi(\hhh)$-module can be extended an irreducible $U_\chi(\bbb)$-module by letting $\nnn^+$ act trivially. By continuously inducing, we define the queer baby Verma module of $U_\chi(\ggg)$ associated with $\lambda\in \hhh_\bz^*$:
\begin{align}\label{eq: q baby}
Z^\qqq_\chi(\lambda)= U_\chi(\ggg)\otimes_{U_\chi(\bbb)} V_\chi(\lambda).
\end{align}
Set  $v_0= 1\otimes 1\in Z_\chi(\lambda)$. We have, as a vector space, $Z^\qqq_\chi(\lambda)=U_\chi(\nnn^-)\otimes V_\chi(\lambda)$

\subsection{A criterion  of irreducible baby Verma modules} Keep the notation $\hhh$ of  the standard Cartan subalgebra with basis
$\{J_i\in\hhh_\bz,\; J'_i\in\hhh_\bo\mid 1\leq i\leq n\}$.  For $\lambda:= (\lambda_1,\ldots,\lambda_n)\in \Lambda_\chi$  with  $\lambda_i = \lambda(J_i)$, put
$$\Phi(\lambda):=\prod_{1\leq i<j\leq n}\phi(\lambda_i,\lambda_j)$$
where the function $\psi$ on two parameters is defined via $\phi(x,y) = (x+ y)(x-y-1)(x - y- 2) \cdots(x- y -(p -1))$.

 \begin{theorem} \cite[Theorem 3.4]{WZ2}\label{thm: wz2 thm q}
  Let $\chi\in \ggg_\bz^*$ be semisimple with $\chi(\nnn^+_\bz) =\chi(\nnn^-_\bz) = 0$ and take $\lambda\in \Lambda_\chi$.
Then the queer baby Verma module $Z^\qqq_\chi(\lambda)$ is irreducible if and only if $\Phi(\lambda)\ne 0$.
\end{theorem}


\section{Proof of  the first super Kac-Weisfeiler conjecture for strange classical Lie superlagebras}

\subsection{General observation on isotropic spaces associated with $\theta\in \ggg_\bz^*$ for any $\ggg=\ggg_\bz\oplus\ggg_\bo$}\label{sec: gen notions for conj}
Let $\ggg=\ggg_\bz\oplus\ggg_\bo$ be any given finite-dimensional restricted Lie superalgebra over $\bk$. For a given $\theta\in \ggg_\bz^*$, consider the bilinear form $\tsb_\theta$ on $\ggg$ with regarding $\theta\in \ggg^*$ by trivial extension
$$\tsb_\theta: \ggg\times \ggg\rightarrow \bk, (X,Y)\mapsto \theta([X,Y]).$$

\subsection{} Let $\ggg^\theta:=\{X\in \ggg\mid \theta([X,\ggg])=0\}$. Then one can define bilinear forms arising from $(-,-)$ on the spaces $\tilde\ggg:=\ggg\slash \ggg^\theta$, $\tilde\ggg_\bz:=\ggg_\bz\slash \ggg_\bz^\theta$ and $\tilde\ggg_\bo:=\ggg_\bo\slash \ggg_\bo^\theta$ respectively. By abuse of notations, those bilinear forms are still denoted by $\tsb_\theta$.

\begin{lemma}(\cite{Shu}) \label{lem: non-deg bil forms}
The following statements hold.
\begin{itemize}\label{lem: 2.1}
\item[(1)] The centralizer $\ggg^\theta=\ggg^\theta_\bz+\ggg^\theta_\bo$ is a restricted subalgebra of $\ggg$.
\item[(2)] The  $\tsb_\theta$ on $\tilde \ggg_\bz$ is a  non-degenerate skew-symmetric bilinear form. And $\tsb_\theta$ on $\tilde\ggg_\bo$ is a non-degenerate symmetric bilinear form.
 Consequently, $\dim (\ggg_\bz-\ggg_\bz^\theta)$ is even.

\item[(3)] The maximal isotropic space with respect to $\tsb_\theta$ in $\ggg_\bz$ has dimension $\frac{\dim\ggg_\bz+\dim\ggg^\theta}{2}$.

 \item[(4)] The maximal isotropic space with respect to $\tsb_\theta$  in $\ggg_\bo$ has dimension $\frac{\dim\ggg_\bo+\dim\ggg^\theta_\bo}{2}$ if $\dim\ggg_\bo-\dim\ggg^\theta_\bo$ is even, and has dimension $\frac{\dim\ggg_\bo+\dim\ggg^\theta_\bo-1}{2}$ if  $\dim\ggg_\bo-\dim\ggg^\theta_\bo$ is odd.
\end{itemize}
\end{lemma}

 In the following,  let $\lceil a\rceil$ denote the  greatest  integer lower bound of $a$ for an rational number $a\in\bbq$, and $\lfloor a\rfloor$ denote the  least integer upper bound of $a$. Then Lemma \ref{lem: 2.1}(4) becomes that the maximal isotropic space with respect to $\tsb_\theta$  in $\ggg_\bo$ has dimension $\lceil\frac{\dim\ggg_\bo+\dim\ggg^\theta_\bo}{2}\rceil$. Set $d(\ggg,\theta)=(\frac{\dim\ggg_\bz+\dim\ggg^\theta_\bz}{2}| \lceil\frac{\dim\ggg_\bo+\dim\ggg^\theta_\bo}{2}\rceil)$. The $d(\ggg,\theta)$ is the maximal super-dimension of the isotropy subspaces of $\ggg$ with respect to $(-,-)$.

\subsection{}\label{sec: skw conj} For the simplicity of expression, a pair of non-negative integer $(a|b)$ is said to be a super-datum. Call $a$ and $b$ its even entry and odd entry, respectively.   For a given $\theta\in\ggg^*_\bz$, denote the super datum $i(\ggg,\theta)=(\frac{\dim\ggg_\bo-\dim\zzz^\theta_\bo}{2}| \lfloor\frac{\dim\ggg_\bo-\dim\zzz^\theta_\bo}{2}\rfloor)$ by
$(b^\theta_0, b^\theta_1)$, this is to say, the even and odd entries are respectively,
\begin{align}\label{eq: theta value}
&b^\theta_0=\dim\ggg_\bz-\dim\zzz^\theta_\bz\cr
 & b^\theta_1=\dim\ggg_\bo-\dim\zzz^\theta_\bo.
\end{align}
And set
\begin{align}\label{eq: 1 KW version}
p^{\frac{b_0}{2}}2^{{\lfloor\frac{b_1}{2}\rfloor}}=\max_{\theta\in \ggg_\bz^*}p^{\frac{b^\theta_0}{2}}2^{{\lfloor\frac{b^\theta_1}{2}\rfloor}}.
\end{align}

In \cite{Shu}, the author proposed the following conjecture.
\begin{conj}\label{conj} Let $\ggg$ be a finite-dimensional restricted Lie superalgebra over $\bk$.
The maximal dimension of irreducible $\ggg$-modules  is  $p^{\frac{b_0}{2}}2^{\lfloor\frac{b_1}{2}\rfloor}$.
\end{conj}

\begin{remark} This conjecture is a super version of the first Kac-Weisfeiler conjecture on irreducible modules of restricted Lie algebras (see \cite{Kac2}). For the latter, the study has progress, but the question is still open (see  \cite{PrSk}, \cite{MST}).
\end{remark}

\subsection{$\ggg=\tilde\ppp(n)$}\label{sec: periplectic sec show} In this subsection we fix $\ggg=\tilde\ppp(n)$. We have the following observation.
\begin{lemma}\label{lem: bigger than n}
The following statements hold.
\begin{itemize}
\item[(1)] For  any  $\theta\in \ggg_\bz^*$, $\dim\ggg_\bo^\theta\geq n$. Consequently, $b^\theta_1\leq n^2-n$.
\item[(2)] If $\chi$ is a regular semisimple $p$-character, then $\dim\ggg_1^\chi=n$. Consequently, $b_1=n^2-n$.

\end{itemize}
\end{lemma}

\begin{proof} (1) Consider ${\tsb_\theta}|_{\ggg_\bo\times\ggg_\bo}: \ggg_\bo\times \ggg_\bo\rightarrow \bk$ mapping $(X,Y)$ onto $\theta([X,Y])$, which gives rise to a symmetric bilinear form on the space $\ggg_\bo$. We denote it by $\tsbo_\theta$.
 Let $\ker(\tsbo_\theta)=\{X\in \ggg_\bo\mid \tsbo_\chi(X,\ggg_\bo)=0\}$. Then $\ker(\tsbo_\theta)$ is exactly equal to the centralizer $\ggg_\bo^\theta$ of $\theta$ in $\ggg_\bo$. We can define a non-degenerate bilinear form arising from $\tsbo_\theta$ on the spaces $\tilde\ggg_\bo:=\ggg_\bo\slash \ggg_\bo^\theta$. By abuse of notation, this non-degenerate bilinear form are still denoted by $\tsbo_\theta$. Hence the dimension of maximal isotropic subspaces of $\tilde\ggg_\bo$ with respect of $\tsbo_\theta$ is  $\lceil\frac{\dim\ggg_\bo-\dim\ggg^\theta_\bo}{2}\rceil$. Consequently,  the dimension of maximal isotropic subspaces of $\ggg_\bo$ with respect of $\tsbo_\theta$ is equal to
 \begin{align*}
 &\lceil\frac{\dim\ggg_\bo-\dim\ggg^\theta_\bo}{2}\rceil+\dim\ggg^\theta_\bo\cr
 =&\lceil\frac{\dim\ggg_\bo+\dim\ggg^\theta_\bo}{2}\rceil\cr
 =&\lceil\frac{n^2+\dim\ggg^\theta_\bo}{2}\rceil.
 \end{align*}

  Now we have $\theta (\ggg_1,\ggg_1)=0$. This  means that $\ggg_1$ is an isotropic subspace of $\ggg_\bo$ with respect to $\tsbo_\theta$. Hence we have the following
 $$ \dim \ggg_1\leq \lceil\frac{n^2+\dim\ggg^\theta_\bo}{2}\rceil$$
 which means ${1\over 2}(n^2+n)\leq \lceil\frac{n^2+\dim\ggg^\theta_\bo}{2}\rceil$. Hence $\dim\ggg_\bo^\theta\geq n$. Consequently, $b^\theta_1\leq n^2-n$.

(2) It suffices to show that $\ggg^\chi_\bo$ coincides with $\mathcal{H}$. We verify this by contradiction. Suppose $\ggg^\chi$ contains properly $\mathcal{H}$. Then there exits nonzero $X=\sum_{i<j}(a_{ij}X_{\epsilon_i+\epsilon_j}+b_{ij}X_{-(\epsilon_i+\epsilon_j)})+\sum_{i=1}^nc_iX_{2\epsilon_i}\in \ggg^\chi$ with $a_{ij}, b_{ij}$ not all zero for $1\leq i<j\leq n$. For example, suppose $a_{ij}\ne0$, then $\chi([X, X_{-(\epsilon_i+\epsilon_j)}])=2a_{ij}\chi(H_i-H_j)\ne 0$. This contradicts the assumption that $X\in \ggg_\bo^\chi$.
Hence $\ggg^\chi_\bz$ must coincides with $\mathcal{H}$. The proof is completed.
\end{proof}

 \subsubsection{} Now we give a key result for periplectic Lie superalgebras.

\begin{prop}\label{thm: first KW for p} The maximal dimension of irreducible $\ggg$-modules  is  $p^{\frac{b_0}{2}}2^{\lfloor\frac{b_1}{2}\rfloor}=(2p)^{n(n-1)\over 2}$.
\end{prop}

\begin{proof} Note that up to the Morita equivalence arising from $\GL(n)$-conjugation,  any irreducible $U(\ggg)$-module $S$ must be an irreducible quotient of $Z_\theta(\mu)$ for some $\theta\in \ggg_\bz^*$ and $\mu\in\Lambda(\theta)$. Hence any irreducible module has dimension less than $\dim Z_\theta(\mu)=\dim U_\theta(\caln^-)=p^{\dim\nnn_\bz^{-}}2^{\dim \ggg_{-1}}$.

By Proposition \ref{prop: reg ss}, $Z_\theta(\lambda)$ is irreducible when $\theta$ is regular semisimple. Hence the maximal dimension of irreducible $U(\ggg)$-module is $p^{\dim\nnn_\bz^{-}}2^{{n(n-1)}\over 2}$.

Furthermore, take $\chi$ to be  regular semisimple. Naturally, it is regular semisimple with respect to $\ggg_\bz\cong \gl(n)$. Hence $\dim \nnn_\bz^-=\max_{\theta\in \ggg_\bz^*}{{\dim \ggg_\bz-\dim \ggg_\bz^\theta}\over 2}={b_0\over 2}$ by a classical result (see \cite{FP88} or \cite{Jan97}). On the other hand, for any  $\theta\in \ggg_\bz^*$ by Lemma \ref{lem: bigger than n}(1) we always have
$\dim\ggg_\bo-\dim\ggg_\bo^\theta\leq n^2-n$.
  In the meanwhile, for the regular semisimple $p$-character $\chi$, $\ggg_\bo^\chi=\sum_{i=1}^n\bk E_{i,n+i}$ by Lemma \ref{lem: bigger than n}(2),
  which means
$\dim\ggg_\bo-\dim\ggg_\bo^\chi=n^2-n$.   Hence we have
\begin{align*}
{\lfloor {b_1\over 2}\rfloor}=\max_{\theta\in \ggg_\bz^*}\lfloor{{{\dim\ggg_\bo-\dim\ggg_\bo^\theta}\over 2}\rfloor}={{n^2-n}\over 2}=\dim \nnn^-_{\bz}.
\end{align*}
Summing up, we have $p^{{b_0\over 2}}2^{\lfloor {b_1\over 2}\rfloor}=p^{\dim\nnn_\bz^-}2^{\dim \nnn^-_{\bz}}=\dim Z_\chi(\lambda)$ for a regular semisimple $p$-character $\chi$ and $\lambda\in \Lambda(\chi)$.

The proof is completed.
\end{proof}

\subsection{$\ggg=\qqq(n)$.} Now let $\chi\in\ggg^*_\bz$ be arbitrarily given.
 %
 Clearly, $\dim\ggg_\bo^\chi\geq 0$ and  $b^\chi_1\leq \dim\ggg_\bo=n^2$.
 Furthermore, we have a more precise  observation, prior to which we  need some preparations. We need to reformulate Definition  \ref{def: regular ss} for regular semisimple $p$-characters in the case of periplectic case. Recall there are basis elements $J_k=E_{kk}+ E_{n+k,n+k}$ for $k=1,\ldots,n$ in $\hhh_\bz$.

 \begin{defn}\label{defn: regular ss for q}
 \begin{itemize}
 \item[(1)] Say that a $p$-character $\theta\in\ggg_\bz^*$ is regular semisimple if under $\textsf{Ad}(g)$ for $g\in\text{GL}(n)$  it is conjugate to
 \begin{align}\label{eq: regular ss q}
 \chi\in\hhh_\bz^*\subset \ggg_\bz^* \text{ such that }\chi(J_k) \text{ are mutually different nonzero number for }  k=1,2,\ldots,n.
  \end{align}
 \item[(2)]
 Say that $\chi\in\ggg^*_\bz$ is strongly regular semisimple for queer Lie superalgebra $\ggg=\qqq(n)$ if $\chi$ is regular semisimple and additionally
 \begin{align}\label{eq: str reg ss}
 \chi(\textsf{Ad}(g)(J_i+J_j))\ne0  \text{  for all }1\leq i<j\leq n
 \end{align}
 for the same $g\in \GL(n)$ as (1).

 \end{itemize}

\end{defn}

\begin{remark} The set of strongly regular semisimple $p$-characters in $\hhh^*$ forms a Zariski dense subset of $\hhh^*$.
\end{remark}



\begin{lemma}\label{lem: odd centl dim big}
 Let $\chi\in\ggg^*_\bz$ be a strongly regular semsimple $p$-character as in Definition \ref{def: regular ss} with $g=$ identity. Then $b_1^\chi=n^2$. Correspondingly, $b_1=n^2$.
\end{lemma}

\begin{proof} Similar to $\tsf_\lambda$ defined in \S\ref{sec: strange baby}, we can define a map $\tsg_\chi: \ggg_\bo\otimes \ggg_\bo\rightarrow \bk$ sending $(X,Y)$ to $\chi([X,Y])$ for $X,Y\in\ggg_\bo$. Then $\tsg_\chi$  becomes a symmetric bilinear form on $\ggg_\bo$. Then $\ker(\tsg_\chi)$ coincides with $\ggg_\bo^\chi$.
We claim that $\tsg_\chi$ is non-degenerate which means that for any nonzero $X\in \ggg_\bo$, there exists $Y\in\ggg_\bo$ such that $\tsg(X,Y)\ne0$.

Recall that $\ggg_\bo$
is another isomorphic copy of $\gl(n)$ under the adjoint action of $\ggg_\bz$. Then $\ggg_\bz$ and $\ggg_\bo$ have correspondingly ``Chevalley  basis elements"
$$\{X_{\alpha_i}, J_k,  Y_{\alpha_i}\mid i=1,\ldots,n-1, k=1,\ldots,n\} \text{ and } \{X'_{\alpha_i}, J'_k,  Y'_{\alpha_i}\mid i=1,\ldots,n-1, k=1,\ldots,n\}$$
respectively, where $\alpha_i=\epsilon_i-\epsilon_{i+1}$,  $i=1,\cdots,n-1$, are simple roots.  For any nonzero $X\in \ggg_\bo$, we can write
$$X=\sum_{i=1}^{n-1} (a_iX'_{\alpha_i}+ b_iY'_{\alpha_i}+\sum_{k=1}^nc_kJ'_{k})$$
with $a_i, b_i,c_k\in\bk$ not all equal to zero for $i=1,\ldots, n-1; k=1,\ldots,n$. Divide arguments into different cases.

 (1) Suppose $a_i\ne 0$,  we can take  $Y=Y'_{\alpha_i}$, then $\chi(X,Y)=a_i\chi(J_i+J_{i+1})\ne0$.

(2) Suppose $b_i\ne 0$,  we can take  $Y=X'_{\alpha_i}$, then $\chi(X,Y)=b_i\chi(J_i+J_{i+1})\ne0$.

(3) Suppose $c_k\ne 0$, we can take $Y=J'_k$. Then $\chi(X,Y)=2c_k\chi(J_k)\ne 0$.

Summing up, we already show that $\tsg_\chi$ is non-degenerate.  Hence, $\ggg_\bo^\chi=0$, and then $b_1^\chi=n^2$.  By the analysis in the beginning of this subsection along the arguments around (\ref{eq: ann n ev +}), we already know that $b_1^\theta\leq n^2$ for all $\theta\in \ggg^*_\bz$. Hence, $b_1=n^2$. Thus we accomplish the proof.
\end{proof}

\subsubsection{} Keep the notations as in Lemma \ref{lem: odd centl dim big}.  In order to verify the $1^{\text{st}}$ sKW conjecture for $\qqq(n)$, we will further suppose $\chi(\nnn_\bz^+) = 0$ without loss of generality, as explained  around (\ref{eq: ann n ev +}). In this case, we have define queer Verma modules $Z^\qqq_\chi(\lambda)=U_\chi(\ggg)\otimes_{U_\chi(\bbb)}V_\chi(\lambda)$ for $\lambda\in \Lambda(\chi)$. Now let us investigate the least number of $\dim\hhh^\lambda_\bo$ when $\chi$ and $\lambda\in \Lambda(\chi)$ vary, which will be denoted by $\ell$.

\begin{lemma}\label{lem: maximal ell}
 $\ell=\lceil{n\over 2}\rceil$. Furthermore, if $\lambda\in \Lambda(\chi)$ for a regular semisimple $\chi$, then $\dim(\hhh_\bo^\lambda)=\ell$.
\end{lemma}

\begin{proof} For  $\chi\in\ggg^*_\bz$ with $\chi(\nnn^+_\bz)=0$, let $\bar\chi$ be the corresponding to the semisimple $p$-character, i.e. $\bar\chi(\hhh_\bz)=\chi(\hhh)$ and $\bar\chi(\nnn_\bz^\pm)=0$. By definition, $\tsf_{\chi}=\tsf_{\bar\chi}$ and  $\Lambda(\chi)=\Lambda(\bar\chi)$.
Consequently, we have $V_\chi(\lambda)=V_{\bar\chi}(\lambda)$ for $\lambda\in \Lambda(\chi)=\Lambda(\bar\chi)$.

Hence, we can only consider semisimple $p$-character $\chi$  when we focus ourselves on the least number of $\dim\hhh^\lambda_\bo$ with $\chi$ ranging through $\ggg^*_\bz$ (but $\chi(\nnn_\bz^+)=0$) and $\lambda\in \Lambda(\chi)$. We proceed with steps.

(1) Consider $\hhh_\bo\slash \ker(\tsf_\lambda)$. Then by the same reason as in Lemma \ref{lem: bigger than n}(2), we have the dimension of maximal isotropic spaces in $\hhh_\bo$ with respect to $\tsf_\lambda$ is $\lceil {{n-\dim \ker(\tsf_\lambda)}\over 2}\rceil$. Hence $\dim \hhh^\lambda_\bo\leq \lceil {n\over 2}\rceil$.

(2) We suppose $\chi$ is regular semisimple as in (\ref{eq: regular ss q}) without loss of generality. Note that  $\chi(J_k)\ne0$ implies $\lambda(J_k)\ne0$. Hence by the same arguments in the arguments (3)  of the proof of Lemma \ref{lem: odd centl dim big}, we can prove that $\tsf_\lambda$ is non-degenerate. Hence $\dim \hhh^\lambda_\bo=\lceil {{\dim\hhh_\bo}\over 2}\rceil=\lceil {n\over 2}\rceil$.

Summing up, we accomplish the proof.
\end{proof}

\subsubsection{}

\begin{prop} \label{lem: hope} Suppose $\chi$ is strongly regular semisimple $p$-character as in (\ref{eq: regular ss q}) satisfying (\ref{eq: str reg ss}) for $g=\id$. Then the following statements hold.
\begin{itemize}
\item[(1)] There exists $\lambda\in \Lambda(\chi)$ such that
$\dim\hhh^\lambda_\bo=\ell$.
\item[(2)] The queer baby Verma module $Z_\chi^\qqq(\lambda)$ is irreducible for  $\lambda\in \Lambda(\chi)$.

\item[(3)] $\dim Z_\chi^\qqq(\lambda)=p^{n(n-1)\over 2} 2^{\lfloor {n^2\over 2}\rfloor}$.
\end{itemize}
\end{prop}

\begin{proof} (1) follows from Lemma \ref{lem: maximal ell}. (3) is a direct consequence of (1) and (2).  We only need to prove (2).

 By the definition of regular semisimple $p$-characters,  $h_{ij}:=\chi(J_i-J_j)\ne 0$ for $1\leq i<j\leq n$. Hence for $\lambda\in \Lambda(\chi)$ we have $$(\lambda_i-\lambda_j)^p-(\lambda_i-\lambda_j)=h_{ij}^p\ne 0.$$
  Hence $\lambda_i-\lambda_j-k$ is nonzero for all $1\leq i<j\leq n$ and $k\in\{1,\ldots,p-1\}$. By the same arguments, the assumption $\chi(J_i+J_j)\ne 0$ entails that $\lambda_i+\lambda_j\ne 0$ for all $i\ne j$. So $\phi(\lambda_i,\lambda_j)\ne0$. It follows that $\Phi(\lambda)\ne0$.  Hence the part (2)  is a direct consequence of  Theorem \ref{thm: wz2 thm q}. The proof is completed.
\end{proof}

\subsection{The main results of this section}
\begin{theorem}\label{thm: kw conj for strange} The first super Kac-Weifeiler conjecture is true for $\tilde\ppp(n)$ and $\qqq(n)$.
%
\end{theorem}

\begin{proof} Note that $b_0^\chi= n(n-1)$ coincides with $\max_{\theta\in\ggg_\bz^*}(b_0^\theta)$ when $\chi$ is regular semisimple (see \cite{FP88} or \cite{Jan97}).  We can proceed to verify.

(1) The part of verification for $\tilde\ppp(n)$ follows from Lemma \ref{lem: bigger than n} and Proposition \ref{thm: first KW for p}.

(2) As to the part $\qqq(n)$, by combining Lemma \ref{lem: odd centl dim big} with Proposition \ref{lem: hope}, $p^{b_0\over 2}2^{\lfloor{b_1\over 2}\rfloor}$ is equal to $p^{n(n-1)\over 2} 2^{\lfloor {n^2\over 2}\rfloor}$. The verification can be done.

The proof of the theorem is completed.
%
%
%
\end{proof}

\vskip20pt
\section{Verification for the case of derived strange classical algebras }
In this concluding section, we concisely demonstrate that the first super Kac-Weifeiler conjecture is still true for  derived strange classical algebras $[\tilde\ppp(n),\tilde\ppp(n)]$ and $[\qqq(n),\qqq(n)]$.

From now on we suppose $\ggg=[\tilde\ppp(n),\tilde\ppp(n)]$ or $[\qqq(n),\qqq(n)]$, as usually in literatures, both of which are  denoted by $\ppp(n)$ and  $\tilde\qqq(n)$ respectively.

\subsection{$\ggg=[\tilde\ppp(n),\tilde\ppp(n)]$}
The same arguments as in \S\ref{sec: periplectic sec show} enables us to have the following counterpart of Proposition \ref{thm: first KW for p}.
\begin{prop}\label{prop: first KW for derirved p}
 Let $\ggg=[\tilde\ppp(n),\tilde\ppp(n)]$. Then the maximal dimension of irreducible $\ggg$-modules  is  $p^{\frac{b_0}{2}}2^{\lfloor\frac{b_1}{2}\rfloor}=(2p)^{n(n-1)\over 2}$.
\end{prop}

Notice again that  $\dim b_0^\chi\leq n(n-1)$ for all $\chi\in\ggg_\bz^*$ and the equality happens  whenever $\chi$ is regular semisimple. Hence $\max_{\theta\in\ggg_\bz^*}(b_0^\theta)$ is equal to $n(n-1)$.   On the other side, Lemma \ref{lem: bigger than n} still true for $\ggg=[\tilde\ppp(n),\tilde\ppp(n)]$. Hence we have

\begin{corollary}\label{cor: 5.2 p} The first super Kac-Weisfeiler conjecture is true for $\ggg=[\tilde\ppp(n),\tilde\ppp(n)]$.
\end{corollary}

\subsection{$\ggg=\tilde\qqq(n)$}
Parallel to Proposition \ref{lem: hope} we have the following result

\begin{prop}\label{prop: first KW for derirved q} Let $\ggg=[\qqq(n), \qqq(n)]$. Then the maximal dimension of irreducible $\ggg$-modules  is  $p^{\frac{b_0}{2}}2^{\lfloor\frac{b_1}{2}\rfloor}=p^{n(n-1)\over 2} 2^{\lfloor{{n^2-1}\over 2}\rfloor}$. Consequently, the first super Kac-Weifeiler conjecture is true for $\ggg$.
\end{prop}

\begin{proof} The proof essentially follows the arguments in the case of $\qqq(n)$. The difference lies in the Cartan subalgebras which influences the final result on maximal dimensions of irreducible modules,  according to the change of   queer baby Verma modules defined as in (\ref{eq: q baby}).

Note that in comparison with $\qqq(n)$,  the Cartan subalgebra $\hhh$ of $\ggg=[\qqq(n),\qqq(n)]$ changes into $\hhh=\hhh_\bz\oplus\hhh_\bo$ with $\hhh_\bo$ admitting dimension $n-1$.
By the same arguments as Lemma \ref{lem: maximal ell}, we can prove that the least one of all dimensions of $\hhh_\bo^\lambda$ with $\chi\in\ggg^*_\bz$ and $\lambda\in\Lambda(\chi)$ is $\lceil{{n-1}\over 2}\rceil$. Moreover, $\dim\hhh_\bo^\lambda=\lceil{{n-1}\over 2}\rceil$ whenever $\chi$ is a strongly regular semisimple $p$-character and $\lambda\in \Lambda(\chi)$.

Consequently, by almost same arguments as in Proposition \ref{lem: hope} with some mild modification we can show that if $\chi$ is strongly regular semisimple $p$-character as in the proposition, then  the queer baby Verma module
$$Z^\qqq_\chi(\lambda)=U_\chi(\ggg)\otimes_{U_\chi(\bbb)}
V_\chi(\lambda)$$ is irreducible for  $\lambda\in \Lambda(\chi)$
and $V_\chi(\lambda) = U_\chi(\hhh)\otimes_{U_\chi(\hhh_\bz\oplus \hhh^\lambda_\bo)}\bk_\lambda$.
In this case,  $$\dim Z_\chi^\qqq(\lambda)=p^{n(n-1)\over 2} 2^{\lfloor {{n^2-1}\over 2}\rfloor}.$$

As in the case $[\tilde\ppp(n),\tilde\ppp(n)]$,  $\dim b_0^\chi\leq n(n-1)$ for any $\chi\in\ggg_\bz^*$ and the equality happens  whenever $\chi$ is regular semisimple. Hence $\max_{\theta\in\ggg_\bz^*}(b_0^\theta)$ is equal to $n(n-1)$. The above analysis already shows that  the maximal dimension of irreducible $\ggg$-modules  is  $p^{\frac{b_0}{2}}2^{\lfloor\frac{b_1}{2}\rfloor}=p^{n(n-1)\over 2} 2^{{n^2-1}\over 2}$. The first super Kac-Weifeiler conjecture is proved  for $\ggg$.
\end{proof}

\end{document}